\documentclass[12pt]{amsart}
\usepackage{amscd}
\usepackage{pstcol,pst-plot,pst-3d}
\def\D{{\Delta}}
\def\G{{\Gamma}}
\def\F{{\mathcal F}}
\def\a{{\bold a}}

\def\c{{\bold c}}

\def\1{{\mathbf 1}}
\def\0{{\mathbf 0}}
\def\opn#1#2{\def#1{\operatorname{#2}}} 
\opn\diam{diam} 
\opn\depth{depth}
\opn\star{star}
\opn\lk{lk}

\newtheorem{Theorem}{Theorem}[section]
\newtheorem{Lemma}[Theorem]{Lemma}
\newtheorem{Corollary}[Theorem]{Corollary}
\newtheorem{Proposition}[Theorem]{Proposition}

\newtheorem{Example}[Theorem]{Example}

\begin{document}

\title{Equality of ordinary and symbolic powers\\ of Stanley-Reisner ideals}

\author{Ngo Viet Trung}
\address{Institute of Mathematics, 18 Hoang Quoc Viet, Hanoi, Vietnam}
\email{nvtrung@math.ac.vn}

\author{Tran Manh Tuan}
\address{Institute of Mathematics, 18 Hoang Quoc Viet, Hanoi, Vietnam}
\email{manhtuankhtn@gmail.com}

\thanks{The first author is partially supported by the National Foundation for Science and Technology Development} 
\keywords{Stanley-Reisner ideal, Cohen-Macaulay ideal, symbolic power, vertex cover, Ramsey graph} 
\subjclass{13C05, 13H10}

\begin{abstract}
This paper studies properties of simplicial complexes $\D$ with the equality  $I_\D^{(m)} = I_\D^m$ for a given $m \ge 2$.
The main results are combinatorial characterizations of such complexes in the two-dimensional case.
It turns out that there exist only a finite number of complexes with this property and that these complexes can be described completely.
As a consequence we are able to determine all complexes for which $I_\D^m$ is Cohen-Macaulay for some $m \ge 2$.
In particular, there are complexes with $I_\D^{(2)} = I_\D^2$ or $I_\D^{(3)} = I_\D^3$ but $I_\D^{(m)} \neq I_\D^m$ for all $m \ge 4$ and that if $I_\D^{(m)} = I_\D^m$ for some $m \ge 4$, then $I_\D^{(m)} = I_\D^m$ for all $m \ge 1$. Similarly, there are complexes for which $I_\D^2$ is Cohen-Macaulay but $I_\D^m$ is not Cohen-Macaulay for all $m \ge 3$ and if $I_\D^m$ is Cohen-Macaulay for some $m \ge 3$, then $I_\D$ is a complete intersection. 
\end{abstract}
\maketitle

\section*{Introduction}

Let $I_\D$ be the Stanley-Reisner ideal of a simplicial complex $\D$.  Given an integer $m \ge 2$, we want to know when $I_\D^m$ is a Cohen-Macaulay ideal.  For that we have to study when $I_\D^{(m)}$ is Cohen-Macaulay and when $I_\D^{(m)} = I_\D^m$, where $I_\D^{(m)}$ denotes the $m$-th symbolic power of $I_\D$. These properties are of interest from both algebraic and combinatorial points of view. They were usually investigated for all (large) powers of an ideal, and if they are satisfied, the ideal enjoys good properties. For instance, if $I_\D^m$ is Cohen-Macaulay for all large $m$, then $I_\D$ must be a complete intersection by \cite{CN} and $I_\D^{(m)} = I_\D^m$ for all $m\ge 1$ if and only if the hypergraph of the minimal nonfaces of $\D$ is Mengerian \cite{HHTZ}. However, little is known about these properties for a sole ideal power. \smallskip

The above problems were first studied for one-dimensional complexes in \cite{MT1}, where one can find combinatorial characterizations for the Cohen-Macaulayness of $I_\D^{(m)}$ and  a complete description of all complexes with $I_\D^{(m)} = I_\D^m$ in terms of the associated graph. 
There is a remarkable distinction between the case $m=2$ and $m\ge 3$ in the sense that there are complexes for which $I_\D^{(2)}$ is Cohen-Macaulay but $I_\D^{(m)}$ is not Cohen-Macaulay for all $m\ge 3$ and that if $I_\D^{(m)}$ is Cohen-Macaulay for some $m\ge 3$, then  $I_\D^{(m)}$ is Cohen-Macaulay for all $m \ge 1$. Similarly, there are also complexes with $I_\D^{(2)} = I_\D^2$ but  $I_\D^{(m)} \neq I_\D^m$ for all $m\ge 3$ and, if  $I_\D^{(m)} = I_\D^m$ for some $m\ge 3$, then  $I_\D^{(m)} = I_\D^m$ for all $m\ge 1$. A similar pattern was also found in \cite{RTY} for the case where $\D$ is a flag complex.  The combinatorial characterizations for the Cohen-Macaulayness of $I_\D^{(m)}$ were subsequently generalized  for simplicial complexes of arbitrary dimension in \cite{MT2}.  The results of \cite{MT1}, \cite{MT2}, \cite{RTY} have raised some general questions for complexes of a given dimension such as \smallskip

\noindent {\bf Question 1}. Is the number of complexes with $I_\D^{(m)} = I_\D^m$ for some $m\ge 2$ finite? \smallskip

\noindent {\bf Question 2}. Does there exist a number $m_0$ such that if $I_\D^{(m)} = I_\D^{m}$ for some $m \ge m_0$, then $I_\D^{(m)} = I_\D^m$ for all $m \ge 1$? \smallskip

To give a positive answer to Question 1 we only need to show that there is an upper bound for the number of the vertices in terms of $\dim \D$. Question 2 is closely related to the stability of associated primes of ideal powers \cite{Bro}. For monomial ideals, this stability has been recently studied in \cite{CM}, \cite{FHT}, \cite{HaM}, \cite{Ho}. From these works one can deduce that there is a number $m_0$ depending on the number of vertices such that if $I_\D^{(m)} = I_\D^{m}$ for some $m \ge m_0$, then $I_\D^{(t)} = I_\D^t$ for all $t \ge m$. However, these works don't provide an answer to Question 2. \smallskip

In this paper we will describe all two-dimensional complexes with $I_\D^{(m)} = I_\D^{m}$ for some $m \ge 2$. As consequences we obtain positive answers 
 to the above questions and we are able to determine all complexes for which $I_\D^m$ is Cohen-Macaulay for some $m \ge 2$.  The main tool is the description of symbolic powers by means of vertex covers of the complex in \cite{HHT}. The paper is divided into three sections.  \smallskip

In Section 1 we carry out preliminary investigations on the equality $I_\D^{(m)} = I_\D^m$. We shall see that this equality imposes strong conditions on the complex $\D$. If $\dim \D =2$, these conditions imply that the graph of the edges of $\D$ is a certain Ramsey graph. From this it follows that the number of the vertices must be very small. Hence there are only a finite number of complexes with $I_\D^{(m)} = I_\D^m$ for some $m \ge 2$. For $\dim \D \ge 3$, we can show that if $I_\D^{(m)} = I_\D^m$ for some $m \ge \dim\D + 2$, then the number of the vertices is bounded by $2(\dim \D+1)$ and that this bound is sharp.\smallskip

In Section 2 we will describe all two-dimensional complexes with $I_\D^{(m)} = I_\D^m$ for some $m \ge 2$. We give a combinatorial characterization of all complexes with $I_\D^{(2)} = I_\D^2$ and we determine all complexes with $I_\D^{(m)} = I_\D^m$ for some $m \ge 3$. It turns out that there are complexes with $I_\D^{(2)} = I_\D^2$ or $I_\D^{(3)} = I_\D^3$ but $I_\D^{(m)} \neq I_\D^m$ for all $m \ge 4$ and that if $I_\D^{(m)} = I_\D^m$ for some $m \ge 4$, then $I_\D^{(m)} = I_\D^m$ for all $m \ge 1$. These results indicate that Questions 1 and 2 may have positive answers in general.  \smallskip

In Section 3 we use results of \cite{MT2} to characterize two-dimensional complexes for which $I_\D^{(2)}$ or all $I_\D^{(m)}$ are Cohen-Macaulay. Combining these characterizations with the results of Section 2 we are able to determine all complexes for which $I_\D^m$ is Cohen-Macaulay for some $m \ge 2$. We shall see that there are complexes for which $I_\D^2$ is Cohen-Macaulay but $I_\D^m$ is not Cohen-Macaulay for all $m \ge 3$ and that if $I_\D^m$ is Cohen-Macaulay for some $m \ge 3$, then $I_\D$ is a complete intersection. These results resemble the results for one-dimensional complexes \cite{MT1} and for flag complexes \cite{RTY}. So it is quite natural to ask the following question. 
\smallskip

\noindent {\bf Question 3}. Is $I_\D$ a complete intersection if $I_\D^m$ is Cohen-Macaulay for some $m \ge 3$? 
\smallskip

The arguments of this paper are a mix between algebraic and combinatorial tools.
They may provide techniques for the study of Questions 1, 2 and 3 and related problems in higher dimensional cases.  
\smallskip

\noindent{\em Acknowledgement}. The authors have been informed by G.\ Rinaldo, N.\ Terai and K.\ Yoshida that
they have studied the Cohen-Macaulayness of the second power of Stanley-Reisner ideals by using a method similar to our method of using Ramsey theory and
that they have found some complexes of Theorem 3.7.

\section{Symbolic powers of Stanley-Reisner ideals}

Let $\D$ be a simplicial complex on the vertex set $[n] := \{1,...,n\}$.
Let $R = K[x_1,...,x_n]$ be a polynomial ring over a field $K$. 
The {\it Stanley-Reisner ideal} of $\D$ is the ideal 
$$I_\D = \big(x_{i_1}\cdots x_{i_s}|\ \{i_1,...,i_s\} \notin \D\big).$$
We will always assume that every vertex appears in $\D$. This means that $I_\D$ is non-degenerate ($I_\D$ doesn't contain linear forms).  \smallskip

For every $F \subset [n]$ let $P_F$ denote the ideal of $R$ generated by the variables $x_i, i \notin F$. Then we have the following decomposition
$$I_\D = \bigcap_{F \in \F(\D)} P_F,$$
where $\F(\D)$ is the set of the facets of $\D$. Since the $m$-th symbolic power $I_\D^{(m)}$ is defined as the intersection of the primary components of $I_\D^m$ associated to the minimal primes of $I_\D$, we have
$$I_\D^{(m)} = \bigcap_{F \in \F(\D)} P_F^m.$$

Let $\D_c$ be the simplicial complex generated by the complements of the facets of $\D$ in $[n]$.
We call a non-negative integral vector $\a = (a_1,...,a_n)$  an {\it $m$-cover} of $\D_c$ if $\sum_{i \in G}a_i \ge m$ for every facet $G$ of $\D_c$.  
Let $x^\a = x_1^{a_1}\cdots x_n^{a_n}$.  The ideal $I_\D^{(m)}$ can be described in terms of $\D_c$ as follows \cite[Section 4]{HHT}. \smallskip

\begin{Lemma}\label{cover 1}
 $x^\a \in I_\D^{(m)}$ if and only if $\a$ is an $m$-cover of $\D_c$. 
\end{Lemma}

Note that $x^\a$ is a squarefree monomial if $\a_i = 0,1$ for all $i = 1,...,n$.
In this case, we may consider $\a$ or $x^\a$ as the set $\{i \in [n]|\ a_i = 1\}$. 
Conversely, we can associate every subset $H \subseteq [n]$ with its incidence vector  
whose $i$-th coordinate equals 1 if $i \in H$ and 0 if $i\not\in H$.  
For this reason we also call $H$ an $m$-cover of $\D_c$ if its incidence vector is an $m$-cover of $\D_c$. \smallskip

It is obvious that $H$ is an $m$-cover of $\D_c$ if and only if $|H \cap G| \ge m$ for every facet $G$ of $\D_c$. This is equivalent to the condition that $H$ contains at least $m$ vertices outside every facet of $\D$. In particular, $H$ is an 1-cover of $\D_c$ if and only if $H$ is not contained in any facet of $\D$. Such a set $H$ is called a {\it nonface} of $\D$. \smallskip

As a consequence, $I_\D$ is generated by the monomials of the 1-covers of $\D_c$.
From this it follows that the monomials of $I_\D^m$ correspond to the sums of $m$ 1-covers of $\D_c$.

\begin{Corollary}\label{cover 2}
$I_\D^{(m)} = I_\D^m$ if and only if every $m$-cover of $\D_c$ is the sum of $m$ 1-covers.
\end{Corollary}

We will use Lemma \ref{cover 1} and Corollary  \ref{cover 2} freely without referring to them. \smallskip

If $\D$ is pure and $\dim \D = n-3$, $\D_c$ is a simple graph. In this case, we know that $I_\D^{(m)} = I_\D^m$ for all $m \ge 1$ if and only if $\D_c$ is bipartite  \cite[Theorem 5.1]{HHT} (the sufficient part was proved in \cite[Corollary 2.6]{GRV}). We will improve this result as follows.

\begin{Proposition}\label{graph} 
Let $\D$ be a pure simplicial complex with $\dim \D = n-3$. Then $I_\D^{(m)}=I_\D^m$ for some $m \geq 2$ (or all $m \ge 2$) if and only if $\D_c$ is a bipartite graph. 
\end{Proposition}

\begin{proof} 
We only need to show the necessary part. Assume that $I_\D^{(m)}=I_\D^m$ for some $m \geq 2$.
If $\D_c$ is not a bipartite graph, it has an induced odd cycle, say on the vertices $1,...,2r+1$, $r \ge 1$. Assume that $\{1,2\},\{2,3\},...,\{2r,2r+1\},\{2r+1,1\}$ are the edges of this cycle.
For every 1-cover $\c =(c_1,\ldots,c_n)$ of $\D_c$ we have $c_i + c_j \ge 1$ if $\{i,j\}$ is an edge of $ \D_c$. Therefore,
\begin{align}
\sum_{i=1}^{2r+1}c_i &= \frac{1}{2}\big[(c_1+c_2)+\cdots+(c_{2r}+c_{2r+1})+(c_{2r+1}+c_1)\big] \notag  \\
&\geq \left\lceil \frac{1}{2}(2r+1)\right\rceil = r+1, \notag
\end{align}
where $\lceil a \rceil$ denotes the smallest integer $\ge a$. Since there are no edges of $\D_c$ connecting the vertices $1,3,...,2r+1$, the vector
$$\a=(\underbrace{1,m-1,\ldots,1,m-1}_{2r},m-1,\ldots,m-1)$$ 
is an $m$-cover of $\D_c$. Therefore, $\a = \c_1 + \cdots + \c_m$ for some 1-covers $\c_1,\ldots,\c_m$. The sum of the first $2r+1$ coordinates of $\a$ is $r + (r+1)(m-1) < (r+1)m$, whereas the sum of the first $2r+1$ coordinates of $\c_1+\cdots+\c_m$ is $\ge (r+1)m$. So we obtain a contradiction.
\end{proof}

In the following we denote by $\D_1$ the graph of the edges of $\D$.

\begin{Lemma} \label{nontriangle}
Assume that $I_\D^{(m)}=I_\D^m$ for some $m \ge 2$. Then  $\D_1$ has no independent set of size 3.
\end{Lemma}

\begin{proof}
If $\D_1$ has an independent set of size 3, say $\{1,2,3\},$ then $\{1,2\},\{1,3\},\{2,3\} \not\in \D.$ Therefore, every facet $F$ of $\D$ doesn't contains at least two vertices in $\{1,2,3\}$. This implies $x_1x_2x_3 \in I_\D^{(2)}$  so that
$$x_1^{m-1}x_2^{m-1}x_3=(x_1x_2)^{m-2}(x_1x_2x_3) \in I_\D^{m-2}I_\D^{(2)}\subseteq \ I_\D^{(m)}.$$
But $x_1^{m-1}x_2^{m-1}x_3 \notin I_\D^m$ because it has degree $2m-1,$ whereas the minimal degree of the elements of $I_\D^m$ is $2m,$ a contradiction.
\end{proof}

\begin{Lemma}\label{join face}
Let $\D$ be a pure simplicial complex. Assume that $I_\D^{(m)}=I_\D^m$ for some $m \ge 2$. Then every face $F$ of $\D$ with $\dim F = \dim \D-1$ is contained in at most 2 facets of $\D$.
\end{Lemma}

\begin{proof}
Assume to the contrary that there is a face $F$ of $\D$ with $\dim F = \dim \D-1$ which is contained in 3 facets of $\D$, say $\{1\} \cup F$, $\{2\} \cup F$, $\{3\} \cup F$. Put $f = x_1x_2^{m-1}x_3^{m-1}\prod_{i \in F}x_i^{m-1}$ and $V = \{1,2,3\} \cup F$. Since every facet of $\D$ doesn't contain at least two vertices of $V$, we can easily check that $f \in I_\D^{(m)}$. Therefore, $f \in I_\D^m.$ Since every nonface of $\D$ in $V$ must contain at least two vertices of $\{1,2,3\}$, every monomial of $I_\D$ in the variables $x_i$, $i \in  V$, must be divisible by $x_1x_2$ or $x_1x_3$ or $x_2x_3.$ Therefore, the divisor of every monomial of $I_\D^m$ in $x_1,x_2,x_3$  has degree at least $2m.$ But the divisor of $f$ in $x_1,x_2,x_3$ is the monomial $x_1x_2^{m-1}x_3^{m-1}$, which has degree $2m-1$. So we obtain a contradiction. 
\end{proof}

In the following we denote by $K^t_r$ ($t \le r$) the simplicial complex of all $t$-subsets of a simplex of $r$ vertices. Note that $K^2_r$ is the complete graph $K_r$.

\begin{Proposition}\label{complete}
Let $\D$ be a pure simplicial complex and $d = \dim \D+1$.\par
{\rm (i)} If $I_\D^{(2)}=I_\D^2$, then $\D$ doesn't contain any $K^{\lfloor d+2/2\rfloor}_{d+2}$.\par
{\rm (ii)} If $I_\D^{(m)}=I_\D^m$ for some $m \ge 3$ and $n \ge d+3$, then $\D$ doesn't contain any $K^{d-1}_{d+1}$.
\end{Proposition}

\begin{proof} 
(i) Consider an arbitrary set of $d+2$ vertices, say $[d+2]$.
Set $f =x_1\cdots x_{d+2}$. Since every facet of $\D$ doesn't contain at least two vertices of $[d+2]$, $f \in I_\D^{(2)}$. Hence $f \in I_\D^2$.
From this it follows that the set $[d+2]$ can be divided into two nonfaces of $\D$.
One of these nonfaces must have cardinality $\le \lfloor d+2/2\rfloor$. Hence 
$\D$ doesn't contain $K^{\lfloor d+2/2\rfloor}_{d+2}$.\par
(ii) Assume for the contrary that $\D$ contains a $K^{d-1}_{d+1}$, say on the vertex set $[d+1]$. 
If there exists a nonface of $d$ vertices in $[d+1]$, say $[d]$, we consider the monomial $g =x_1^{m-2}\cdots x_d^{m-2}x_{d+1}x_{d+2}x_{d+3}$. Using the fact that every facet of $\D$ doesn't contain at least a vertex in $[d]$ and at least three vertices in $[d+3]$, we can check that $g \in I_\D^{(m)}$. Hence $g \in I_\D^m$. So $g$ is the product of $m$ monomials in $I_\D$.  At least $m-2$ of these monomials involve only the variables $x_1,...,x_{d+1}$. Since all subsets $F \subset [d+1]$ with $|F| \le d-1$ are faces of $\D$, these monomials have degree $\ge d$. From this it follows that $\deg g \ge d(m-2)+4$. Since $\deg g = d(m-2) + 3$, we obtain a contradiction. Thus, every $d$-set of $[d+1]$ is a facet of $\D$. Set $h = x_1^{m-2} \cdots x_{d+1}^{m-2}x_{d+2}x_{d+3}$. Using the fact that every facet of $\D$ doesn't contain at least three vertices of $[d+3]$, we can easily check that $h \in I_\D^{(m)}$. Hence $h \in I_\D^m$. So $h$ is the product of $m$ monomials in $I_\D$. At least $m-2$ of these monomials involve only the variables $x_1,...,x_{d+1}$. Since all subsets $F \subset [d+1]$ with $|F| \le d$ are faces of $\D$, these monomials have degree $\ge d+1$. From this it follows that $\deg h \ge (d+1)(m-2) + 4$. Since $\deg h = (d+1)(m-2) + 2$, we obtain a contradiction.
\end{proof}

Applying Proposition \ref{complete} to the case $\dim \D = 2$,  we see that the condition $I_\D^{(m)} = I_\D^m$ for some $m \ge 2$ implies that the graph $\D_1$ does not contain any complete subgraph $K_5$ if $m = 2$ or $K_4$ if $m > 2$ and $n \ge 6$.
Note that a complete subgraph is also called a clique. Together with Lemma \ref{nontriangle}, this leads us to the notion of Ramsey graphs. \smallskip

Recall that a {\it Ramsey $(s,t)$-graph} is a graph with no clique of size $s$ and no independent set of size $t$. Ramsey's theorem \cite{Ram} tells us that there are only a finite number of Ramsey $(s,t)$-graphs for each $s$ and $t$ (see \cite{Ra} for a survey on the largest number of vertices of a Ramsey $(s,t)$-graph). 

\begin{Corollary} \label{Ramsey}
Let $\D$ be a pure two-dimensional simplicial complex. \par
{\rm (i)} If $I_\D^{(2)}=I_\D^2$, then $\D_1$ is a Ramsey $(5,3)$-graph.\par
{\rm (ii)} If $I_\D^{(m)}=I_\D^m$ for some $m \geq 3$ and $n \ge 6$, then $\D_1$ is a Ramsey $(4,3)$-graph.
\end{Corollary}

It is known that $n \le 13$ if $\D_1$ is a Ramsey $(5,3)$-graph and that $n \le 8$ if $\D_1$ is a Ramsey $(4,3)$-graph \cite{GG}. From this it follows that there are only a finite number of two-dimensional complexes with $I_\D^{(m)}=I_\D^m$ for some $m \geq 2$.
The same phenomenon also holds in the case $\dim \D = 1$ \cite{MT1}. Therefore, it is quite natural to ask whether there is a bound on $n$ if $I_\D^{(m)}=I_\D^m$ for some $m \geq 2$ in the case $\dim \D > 2$. There is the following partial answer to this problem. 

\begin{Proposition}\label{finite}
Let $\D$ be a pure simplicial complex and $d = \dim \D+1$. If $I_\D^{(m)}=I_\D^m$ for some $m \geq d+1$, then $n \le 2d$. 
\end{Proposition}

\begin{proof}
Assume for the contrary that $n \ge 2d+1$. Let $r$ be the minimal degree of the generators of $I_\D$.
Then there is a nonface of $\D$ of size $r$, say $\{1,...,r\}$.
Let $f = (x_1\cdots x_r)^{m-d}x_{r+1}\cdots x_{2d+1}$. 
Since the complement of every facet of $\D$ contains at least a vertex in $\{1,...,r\}$ and $d+1$ vertices in $[2d+1]$, we can easily check that $f \in I_\D^{(m)}$, which implies $f \in I_\D^m$. From this it follows that $\deg f \ge rm$. So we get the inequality
$r(m-d) + (2d+1-r) \ge rm$, which implies $r(d+1) \le 2d+1$, a contradiction because $r \ge 2$.
\end{proof}

The bound of Proposition \ref{finite} is the best possible.

\begin{Example}
{\rm Let $\D$ be the simplicial complex on $2d$ vertices with two facets $\{1,...,d\}$ and $\{d+1,...,2d\}$. Then $I_\D$ is generated by the monomials $x_ix_j$, $i = 1,...,d,\ j = d+1,...,2d$. These monomials correspond to the edges of a bipartite graph. Hence $I_\D^{(m)}=I_\D^m$ for all $m \ge 2$ by \cite[Theorem 5.9]{SVV}.}
\end{Example}

We now prepare some properties of Ramsey (4,3)-graphs which we shall need later in our investigation. For a graph $\G$ we denote by $\overline \G$ the graph of the nonedges of $\G$.
Note that an independent set of $\G$ of size $t$ is just a complete subgraph $K_t$ of $\overline \G$.

\begin{Proposition} \label{Ramsey properties}
Let $\G$ be a Ramsey (4,3)-graph on $n$ vertices.\par
\par {\rm (i) } If $n =7$, $\overline \G$ has an induced cycle of length 5 or 7.
\par {\rm (ii) } If $n=8$, $\overline \G$ has an induced cycle of length 5.
\end{Proposition}

\begin{proof}
Let $n =7,8$. Assume for the contrary that $\overline{\G}$ contains no induced cycles of length 5 or 7.
By the assumption, $\overline{\G}$ has no cycles of length 3. Hence $\overline{\G}$ has no odd cycles. Thus, $\overline \G$ is a bipartite graph. As a consequence, the vertex set can be divided into two parts such that the induced subgraph of $\G$ on each part is a complete graph. One of these two parts must have at least 4 vertices so that $\G$ contains $K_4$, a contradiction to the assumption. \par
Let $n=8$. If $\overline\G$ has no induced cycle of length 5, it has an induced cycle of length 7, say on the ordered vertices $1,...,7$. 
Since three non-adjacent vertices of this cycle form a triangle of $\G$, all vertices 1,...,7 are vertices of a triangle of  $\G$ not containing 8. 
Since $\G$ doesn't contain $K_4$, the vertex 8 can't be adjacent to all vertices 1,...,7. Assume that $\{1,8\} \not\in \G$.
Since $\overline\G$ does not contain $K_3$, $\{2,8\}, \{7,8\} \in \G$. Hence $\{2,7,8\}$ is a triangle of $\G$. Since $\{2,4,7\},\{2,5,7\}$ are triangles of $\G$, we must have $\{4,8\},\{5,8\} \notin \G$. From this it follows that $\{4,5,8\}$ is a triangle of $\overline{\G}$, a contradiction.
\end{proof}

\section{Criteria for the equality of ordinary and symbolic powers}

In this section we will describe all pure two-dimensional simplicial complexes $\D$ with $I_\D^{(m)} = I_\D^m$ for some $m\ge 2$. \par

If $n = 3,4,5$, we can easily test the condition $I_\D^{(m)} = I_\D^m$ for every $m\ge 2$. In fact, if $n = 3$, then  $\D$ is a simplex and $I_\D = 0$. If $n = 4$, then $I_\D$ is a principal ideal so  $I_\D^{(m)} = I_\D^m$ for all $m\ge 2$.\par

\begin{Theorem}\label{5}  
Let $\D$ be a pure two-dimensional simplicial complex on $5$ vertices. Then  $I_\D^{(m)}=I_\D^m$ for some $m \geq 2$ (or all  $m \geq 2$) if and only if the vertex set can be divided into two nonfaces of two and three vertices. 
\end{Theorem}

\begin{proof} 
By Proposition \ref{graph}, $I_\D^{(m)}=I_\D^m$ for some $m \geq 2$ (or all  $m \geq 2$) if and only if $\D_c$ is a bipartite graph. This means that the vertex set can be divided into two nonfaces. Since the vertex set has 5 elements, these nonfaces have two and three vertices.
\end{proof}

For $n \ge 6$ we have different criteria for  $I_\D^{(m)}=I_\D^m$ when $m = 2$, $m=3$ and $m \ge 4$.

\begin{Theorem}\label{case 2}
Let $\D$ be a pure two-dimensional simplicial complex on $n\geq 6$ vertices. Then $I_\D^{(2)}=I_\D^2$ if and only if $\D$ satisfies the following conditions:
\par {\rm (i) } $\overline\D_1$ does not contain ${K_3}.$
\par {\rm (ii) } If there are 4 vertices, say 1,2,3,4 such that $\{1,2,3\},\{1,2,4\},\{1,3,4\}$ $\notin \D$, then one of the edges $\{1,2\},\{1,3\},\{1,4\}$ doesn't belong to $\D$.
\par {\rm (iii) } If there are 4 vertices, say 1,2,3,4 such that $\{1,2,3\},\{1,2,4\},\{1,3,4\},\{2,3,4\}$ $ \notin \D$, then the set $\{1,2,3,4\}$ can be divided into two nonfaces of two vertices. 
\par {\rm (iv) } Every set of 5 vertices can be divided into two nonfaces of two and three vertices. 
\end{Theorem}

\begin{proof}  Assume that $I_\D^{(2)}=I_\D^2$. Then (i) follows from Lemma \ref{nontriangle}. 
For (ii) we set $f = x_1^2x_2x_3x_4$. Then $f \in I_\D^{(2)}=I_\D^2$. Therefore, $f$ is divisible by a monomial of degree 2 containing $x_1$ in $I_\D$. This monomial must be one of the three monomials $x_1x_2,x_1x_3,x_1x_4$. Hence one of the edges $\{1,2\},\{1,3\},\{1,4\}$ doesn't belong to $\D$.
For (iii) we set $f = x_1x_2x_3x_4$. Since every facet $F$ of $\D$ contains at most two vertices in $\{1,2,3,4\}$, $f \in P_F^2$. Hence $f \in I_\D^{(2)} = I_\D^2$. Therefore, $f$ is a product of two monomials of degree 2 in $I_\D$. We may assume that $x_1x_2,x_3x_4 \in I_\D$. Then $\{1,2\},\{3,4\} \notin \D$.
For (iv) we first note that every 5-set of vertices is a 2-cover of $\D_c$. Therefore, it can be divided into two nonfaces of two and three vertices.  \smallskip

Now assume that $\D$ satisfies the conditions (i), (ii), (iii) and (iv). We have to show that every monomial $f$ of $I_\D^{(2)}$ also belongs to $I_\D^2.$ We distinguish four cases.\smallskip

{\it Case 1:} $f$ involves only two variables, say $x_1,x_2.$  Then $\{1,2\} \notin \D_1.$ 
Therefore, there exists a facet $F\in \F(\D)$
such that $1 \in F$ and $\{1,2\} \not\subset F.$ 
Since $f\in P_{F}^2$ and $P_{F}$ does not contain $x_1,$ $f$ is divisible by $x_2^2.$ Similarly, $f$ is divisible by $x_1^2.$ Hence $f$ is divisible by $(x_1x_2)^2$ so that $f \in I_\D^2.$ \smallskip

{\it Case 2:} $f$ involves only three variables, say $x_1,x_2,x_3.$ Then $\{1,2,3\} \notin \D.$ By (i), we may assume that $\{1,2\} \in \D.$ Let $F$ be an arbitrary facet of $\D$ containing $\{1,2\}.$ 
Since $f\in P_{F}^2$ and since $P_{F}$ does not contain $x_1,x_2,$ $f$ is divisible by $x_3^2.$ \par
If $\{1,3\}, \{2,3\} \in \D_1,$ then we argue as above to see that $f$ is also divisible by $x_1^2, x_2^2.$ Therefore, $f$ is divisible by $(x_1x_2x_3)^2,$ which implies $f\in I_\D^2.$ \par
If $\{1,3\} \notin \D_1$ and $\{2,3\} \in \D_1$,  then $x_1x_3 \in I_\D$ and, similarly, $f$ is divisible by $x_1^2$. Hence $f$ is divisible by $(x_1x_3)^2$, which implies $f \in I_\D^2.$ \par
If $\{1,3\} \in \D_1$ and $\{2,3\} \notin \D_1$,  then $f$ is divisible by $x_2^2$ and $x_2x_3 \in I_\D$. Hence $f$ is divisible by $(x_2x_3)^2$, which implies $f \in I_\D^2.$ \par
If $\{1,3\}, \{2,3\} \notin \D_1,$ then $x_1x_3,x_2x_3 \in I_\D$. Since $f$ is divisible by $x_1x_2x_3^2$,  $f \in I_\D^2.$ \smallskip

{\it Case 3:} $f$ involves only four variables, say $f=x_1^{a_1}x_2^{a_2}x_3^{a_3}x_4^{a_4}$ with $a_1\geq a_2 \geq a_3 \geq a_4 \geq 1.$ \par
If $a_4\geq 2,$ then $f$ is divisible by $(x_1x_2x_3x_4)^2$. Since every 4-set of vertices is a cover of $\D_c$, $x_1x_2x_3x_4 \in I_\D$. Hence $f \in I_\D^2.$\par
If $a_3 >a_4=1,$ then $f \notin P_F^2$ for $F = \{1,2,3\}$. So $\{1,2,3\} \notin \D.$ Hence $x_1x_2x_3 \in I_\D$. Since $f$ is divisible by $x_1^2x_2^2x_3^2$, $f\in I_\D^2.$ \par
If $a_2 > a_3=1,$  we have, similarly, $\{1,2,3\}, \{1,2,4\} \notin \D.$ Hence $x_1x_2x_3,$ $x_1x_2x_4 \in I_\D$. Since $f$ is divisible by $x_1^2x_2^2x_3x_4$, $f\in I_\D^2.$ \par
If $a_1 > a_2 =1,$ then $\{1,2,3\}, \{1,2,4\}, \{1,3,4\} \notin \D.$ By (ii) we may assume that $\{1,2\} \notin \D$. Then $x_1x_2 \in I_\D$. Therefore, $f$ is divisible by $x_1^2x_2x_3x_4 = (x_1x_2)(x_1x_3x_4)$, which implies $f \in I_\D^2$. \par
If $a_1 = 1$, then $\{1,2,3\}, \{1,2,4\}, \{1,3,4\},\{2,3,4\} \notin \D.$ By (iii) we may assume that $\{1,2\},\{3,4\} \notin \D$. Then $x_1x_2,x_3x_4 \in I_\D$. Since $f = x_1x_2x_3x_4$, $f \in I_\D^2$. \smallskip

{\it Case 4:} $f$ involves five or more variables, say $f$ is divisible by $x_1x_2x_3x_4x_5.$ 
By (iv) we may assume that $\{1,2\},\{3,4,5\} \notin\D$. Then $x_1x_2,x_3x_4x_5 \in I_\D$. Hence $f\in I_\D^2.$
\end{proof}

For $m \ge 3$ we first have to study the case $n=6$.

\begin{Theorem}\label{6}
Let $\D$ be a pure two-dimensional simplicial complex on $6$ vertices. Then $I_\D^{(m)}=I_\D^m$ for some $m \ge 3$ (or all $m \ge 1$) if and only if 
$\overline{\D_1}$ contains three disjoint edges and $\D_1$ contains two disjoint triangles.
\end{Theorem}

\begin{proof} 
Assume that $I_\D^{(m)}=I_\D^m$ for some $m \ge 3$. 
By Corollary \ref{Ramsey}, $\D_1$ doesn't contain $K_4$.
Therefore, there is at least an edge, say $\{1,2\} \notin \D$, which implies $x_1x_2 \in I_\D$.
Let $f = x_1^{m-2}x_2^{m-2}x_3x_4x_5x_6$. Note that $x_1\cdots x_6 \in I_\D^{(3)}$. Then $f \in I_\D^{m-3}I_\D^{(3)} \subseteq I_\D^{(m)}$.
From this it follows that $f \in I_\D^m$. Since $\deg f = 2m$, $f$ is the product of $m$ monomials of degree 2 in $I_\D$. Up to a permutation of the indices $3,4,5,6$, there are only the following three such decompositions of $f$.\par

If $f=(x_1x_2)^{m-2}(x_3x_4)(x_5x_6)$, then $\{1,2\},$ $\{3,4\}, \{5,6\}$ are three disjoint edges of $\overline{\D_1}$. \par

If $f=(x_1x_2)^{m-3}(x_1x_3)(x_2x_4)(x_5x_6)$, then  $\{1,3\},$ $\{2,4\},$ $\{5,6\}$ are three disjoint edges of $\overline{\D_1}$. \par

If $f=(x_1x_2)^{m-4}(x_1x_3)(x_1x_4)(x_2x_5)(x_2x_6)$, then $\{1,3\},\{1,4\},\{2,5\},\{2,6\} \in \overline{\D_1}$. By Corollary \ref{Ramsey},  $\overline{\D_1}$ does not contain $K_3.$  Hence $\{3,4\}, \{5,6\} \in \D_1.$ If $\overline{\D_1}$ doesn't contain three disjoint edges, $\{3,5\},$ $\{3,6\},$ $\{4,5\},$ $\{4,6\} \in \D.$  Therefore, $\D_1$ contains the complete graph on the vertices $3,4,5,6$, a contradiction to the fact that $\D_1$ doesn't contain $K_4$. \par

So we have shown that $\overline{\D_1}$ contain three disjoint edges. 
Let $\{1,4\},$ $\{2,5\},$  $\{3,6\}$ be three disjoint edges of $\overline{\D_1}$, that is $\{1,4\},\{2,5\}, \{ 3,6\} \notin \D$. Without restriction we may assume that $\{1,2,3\} \in \D.$ \par

Assume that $\D_1$ doesn't have two disjoint triangles. 
Then $\{4,5,6\}$ is not a triangle of $\D_1$. Hence we may assume that $\{4,6\} \notin \D$. Since $\overline{\D_1}$ doesn't contain $K_3$, $\{1,6\}, \{3,4\} \in \D$. If $\{4,5\} \notin \D$, we also have $\{1,5\},\{2,4\},\{5,6\} \in \D$. From this it follows that $\{1,5,6\},\{2,3,4\}$ are two disjoint triangles of $\D_1$, a contradiction. So we must have $\{4,5\} \in \D$.
The facet of $\D$ containing $\{4,5\}$ must be $\{3,4,5\}$. Hence $\{1,2,6\}$ isn't a triangle of $\D_1$. From this it follows that $\{2,6\} \notin \D$. Similarly, $\{2,4\} \notin \D$.
Hence $\{2,4,6\}$ is a triangle of $\overline{\D_1}$, a contradiction.
So we have shown that  $\D_1$  has two disjoint triangles. This completes the proof for the necessity. \smallskip

For the sufficiency we may assume that $\{1,4\},$ $\{2,5\},$  $\{3,6\}$ are disjoint edges of $\overline{\D_1}$ and $\{ 1,2,3\},$ $\{4,5,6\}$ are disjoint triangles of $\D_1.$ Let $f$ be an arbitrary monomial of $I_\D^{(m)}, m \ge 2$. We have to prove that $f \in I_\D^m.$  \par

Assume that $f$ is divisible by a monomial $g$ of the form $x_1x_4,$ $x_2x_5$ or $x_3x_6$. Since $g$ corresponds to an 1-cover of $\D_c$ which meets every facet of $\D_c$ at only one vertex, $f/g$ must correspond to an $m-1$ cover of $\D_c.$ Therefore, $f/g \in  I_{\D}^{(m-1)}$. By induction we may assume that $f/g \in I_\D^{m-1}$. From this it follows that $f \in I_\D^m$. \par

Assume that $f$ is not divisible by any of the monomials  $x_1x_4,$ $x_2x_5,$ $x_3x_6.$ Then $f$ involves at most three variables. \par

If $f$ involves only two variables, say $x_i^a x_j^b$.
Then $x_ix_j \in I_\D,$ i.e. $\{i,j\} \notin \D_1.$ 
Let $F$ be a facet of $\Delta$ such that $j\in F$ and $i \notin F.$
Since $f\in P_{F}^{m}$ and since $x_j \notin P_{F}$, $f$ is divisible by $x_i^m.$ 
Similarly, $f$ is also divisible by $x_j^m.$ Therefore, $f$ is divisible by $(x_ix_j)^m,$ which implies $f\in I_{\Delta}^m.$ \par

If $f$ involves three variables, we may assume that these variables are $x_1,x_2,x_3$ or $x_1,x_2,x_6$. \par

If $f=x_1^{a_1}x_2^{a_2}x_3^{a_3}$, then $x_1x_2x_3 \in I_\D$.  
Since $\{2,3\} \in \D$, there exist a facet $F$ of $\D$ containing $\{2,3\}$. We have $x_2x_3 \notin P_F$. 
Since $f \in P_{F}^m$, this implies $x_1^{a_1} \in P_F^m$. Hence $a_1 \ge m$.
Similarly, we also have $a_2 \geq m$ and $a_3\geq m.$ 
Therefore, $f$ is divisible by $(x_1x_2x_3)^m$ so that $f \in I_\D^m.$ \par

If  $f=x_1^{a_1}x_2^{a_2}x_6^{a_6}$, then $x_1x_2x_6 \in I_\D$. Since $\{1,2\} \in \D$, we can show similarly that $a_6 \ge m$.
If $\{1,6\}, \{2,6\} \in \D$, we also have $a_1, a_2 \geq m.$
Hence $f$ is divisible by $(x_1x_2x_6)^m \in I_\D^m.$
If $\{1,6\} \in \D$ and $\{2,6\} \notin \D$, then $a_2, a_6 \geq m$ and $x_2x_6 \in I_\D$. Thus, $f$ is divisible by $(x_2x_6)^m \in I_\D^m.$ 
Similarly, if $\{1,6\} \notin \D$ and $\{2,6\} \in \D$, then
$f $ is divisible by $(x_1x_6)^m \in I_\D^m.$ 
If $\{1,6\}, \{2,6\} \notin \D$, then $\{1,6\},\{2,6\} \in I_\D$. 
Let $F$ be a facet of $\D$ containing the vertex 6.  Since $x_6 \notin P_F$ and $f \in P_F^m$, $x_1^{a_1}x_2^{a_2} \in P_F^m$. 
Therefore, $a_1+a_2 \geq m$.  Without restriction we may assume that $a_1+a_2 = m$ and $a_6=m.$
Then $f$ is divisible by $(x_1x_6)^{a_1}(x_2x_6)^{a_2} \in I_\D^{a_1+a_2}=I_\D^m.$ \par
So we always have $f \in I_\D^m,$ as desired.
\end{proof}

Using Theorem \ref{6} we can easily construct complexes with $I_\D^{(2)}= I_\D^2$ but $I_\D^{(m)}\neq I_\D^m$ for all $m \ge 3$. 

\begin{Example}
{\rm Let $\D$ be the complex with
$$\F(\D) = \big\{\{1,2,3\},\{1,2,4\},\{1,3,5\},\{1,4,6\},\{1,5,6\},\{4,5,6\}\big\}.$$

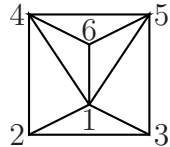
\begin{center}
\psset{unit=0.8cm}
\begin{pspicture}(-0.3,-0.3)(2.3,2.3)
\pspolygon(0,0)(1,0.5)(0,2)
\pspolygon(0,2)(1,1.5)(2,2)
\pspolygon(1,0.5)(2,0)(2,2)
\psline(1,0.5)(1,1.5)
\psline(0,0)(2,0)
\rput(-0.2,0){2}
 \rput(2.2,0){3}
 \rput(1,0.25){1}
 \rput(1,1.75){6}
 \rput(-0.2,2){4}
 \rput(2.2,2){5}
\end{pspicture}
\end{center} 

\noindent It is easy to check that $\D$ satisfies the conditions of Theorem \ref{case 2}. Hence $I_\D^{(2)} = I_\D^2$. Since the vertex 1 is adjacent to all other vertices, $\overline{\D_1}$ doesn't have three disjoint edges. Hence $I_\D^{(m)}\neq I_\D^m$ for all $m \ge 3$.}
\end{Example}

\begin{Theorem}\label{case 3}
Let $\D$ be a pure two-dimensional simplicial complex on $n\geq 6$ vertices. Then  $I_\D^{(3)}=I_\D^3$ if and only if one of the following conditions is satisfied:\par 
{\rm (i) } $n=6,$ $\overline{\D_1}$ has three disjoint edges and $\D_1$ has two disjoint triangles. \par
{\rm (ii) } $n=7$ and up to a permutation of the variables,
$$I_\D=(x_1x_2,x_2x_3,x_3x_4,x_4x_5,x_5x_6,x_6x_7,x_7x_1).$$
\end{Theorem}

\begin{proof} 
Assume that $I_\D^{(3)}=I_\D^3$. Then $\D_1$ is a Ramsey (4,3)-graph by Corollary \ref{Ramsey}.  Hence $n \le 8$ \cite{Ra}.\par
If $n = 6$, then (i) is satisfied by Theorem \ref{6}. \par

If $n \ge  7$, then $\overline{\D_1}$ has an induced cycle of length 5 or 7 by Proposition \ref{Ramsey properties}. If $\overline{\D_1}$ has an induced cycle of length 5, say on the ordered vertices $\{1,2,3,4,5\}$,  then $x_1\cdots x_5 \in I_\D^{(3)}$. It follows that $x_1\cdots x_5 \in I_\D^3$. Hence $\deg x_1\cdots x_5 \ge 6$, a contradiction. So $\overline{\D_1}$ has an induced cycle of length 7, say on the ordered vertices $1,...,7$. Moreover, we must have $n = 7$ by Proposition \ref{Ramsey properties}. 

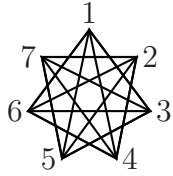
\begin{center}
\psset{unit=0.9cm}
\begin{pspicture}(-1.4,-0.9)(1.4,1.3)
\pspolygon(0,1)(-0.9,-0.2)(0.9,-0.2)
\pspolygon(-0.7,0.6)(0.7,0.6)(0.4,-0.9)
\pspolygon(-0.7,0.6)(0.7,0.6)(-0.4,-0.9)
\pspolygon(0,1)(-0.9,-0.2)(0.4,-0.9)
\pspolygon(0,1)(0.9,-0.2)(-0.4,-0.9)
\pspolygon(-0.7,0.6)(-0.4,-0.9)(0.9,-0.2)
\pspolygon(0.7,0.6)(0.4,-0.9)(-0.9,-0.2)
 \rput(0,1.25){1}
 \rput(0.9,0.6){2}
 \rput(1.1,-0.2){3}
 \rput(0.6,-0.9){4}
 \rput(-0.6,-0.9){5}
 \rput(-1.1,-0.2){6}
 \rput(-0.9,0.6){7}
\end{pspicture}
\end{center} 

\noindent Since $\{1,5\} \in \D$, there is a facet of $\D$ containing $\{1,5\}$. It is easy to see that this facet must be $\{1,3,5\}$. 
Similarly, we also have 
$$\{1,3,6\}, \{1,4,6\}, \{2,4,6\}, \{2,4,7\}, \{2,5,7\},\{3,5,7\} \in \D.$$
Clearly, these are all possible facets for $\D$. Hence
$$\F(\D) = \big\{\{1,3,5\},\{1,3,6\},\{1,4,6\}, \{2,4,6\}, \{2,4,7\}, \{2,5,7\},\{3,5,7\}\big\}.$$
From this it follows that
 $I_\D=(x_1x_2,x_2x_3,x_3x_4,x_4x_5,x_5x_6,x_6x_7,x_7x_1).$ \par

For the sufficiency we assume that $\D$ satisfies one of the conditions (i) and (ii).
For (i) we have  $I_\D^{(3)} = I_\D^3$ by Theorem \ref{6}. For (ii) we note first that $I_\D$ is the edge ideal of a cycle of length 7 so that 
$I_\D^{(3)} = I_\D^3$ by \cite[Lemma 3.1]{CM}.
\end{proof}

\begin{Theorem}\label{case 4}
Let $\D$ be a pure two-dimensional simplicial complex on $n\geq 6$ vertices. Then  $I_\D^{(m)}=I_\D^m$ for some $m\geq 4$ (or all $m \ge 1$) if and only if $n=6,$ $\overline{\D_1}$ has three disjoint edges and $\D_1$ has two disjoint triangles.
\end{Theorem}

\begin{proof}
By Theorem \ref{6} it suffices to show $n = 6$ if $I_\D^{(m)}=I_\D^m$ for some $m\geq 4$. 
But this follows from Proposition \ref{finite}.
\end{proof}

As a consequence, case (ii) of Theorem \ref{case 3} yields a complex with  $I_\D^{(3)}= I_\D^3$ but $I_\D^{(m)}\neq I_\D^m$ for all $m \ge 4$. In general, if $I_\D$ is the edge ideal of a cycle of length $2t+1$, then  $I_\D^{(m)}= I_\D^m$ for $m \le t$ and $I_\D^{(m)}\neq I_\D^m$ for $m \ge t+1$ \cite[Lemma 3.1]{CM}. So we may expect that if $I_\D^{(m)}= I_\D^m$ for some $m \ge \dim\D+2$, then $I_\D^{(m)}= I_\D^m$ for all $m\ge 1$.
\smallskip

On the other hand, we have the following result on the preservation of the equality $I_\D^{(m)}= I_\D^m$.

\begin{Corollary}\label{preservation}
Let $\D$ be a pure two-dimensional simplicial complex.\par
{\rm (i) } If $I_\D^{(m)}=I_\D^m$ for some $m \geq 3$, then $I_\D^{(k)}=I_\D^{k}$ for all $k \le m$. \par
{\rm (ii) } If $I_\D^{(m)}=I_\D^m$ for some $m \ge 4$, then $I_\D^{(m)}= I_\D^{m}$ for all $m \ge 1$. 
\end{Corollary}

\begin{proof}  
If $n=3,$ $I_\D=0.$  If $n=4,$ $I_\D$ is a principal ideal.  
If $n=5,6$, the statements follows from Theorem \ref{5} and Theorem \ref{6}.  
If $n \geq 7$,  $I_\D^{(m)} \neq I_\D^m$ for $m \ge 4$ by Theorem \ref{case 4}. It remains to check 
whether $I_\D^{(2)}= I_\D^2$ in the case of Theorem \ref{case 3} (ii). 
It is easy to see that this case satisfies the conditions of Theorem \ref{case 2}.
\end{proof}

\section{Cohen-Macaulayness of symbolic and ordinary powers}

We first recall the general characterizations of complexes for which $I_\D^{(m)}$ is a Cohen-Macaulay ideal for some $m \ge 2$. \smallskip

Let $\D$ be a simplicial complex on the vertex set $[n]$. One calls $\D$ a {\it Cohen-Macaulay complex} (over $K$) if $I_\D$ is a Cohen-Macaulay ideal. For $F \in \D$ we set 
$$\lk  F = \{G \in \D|\ G \cap F = \emptyset, G \cup F \in \D\}$$
and call it the link of $F$ in $\D$. It is known that $\D$ is Cohen-Macaulay if and only if the reduced cohomology $\tilde H_j(\lk  F,K) = 0$ for all $F \in \D$, $j < \dim \lk  F$ (see e.g. \cite{BrH}). \smallskip

For every subset $V \subseteq [n]$ we denote by $\D_V$ the subcomplex of $\D$ the facets of which are the facets of $\D$ with at least $|V|-1$ vertices in $V$.

\begin{Theorem}\label{second 1} \cite[Theorem 2.1]{MT2}
$I_\D^{(2)}$ is a Cohen-Macaulay ideal if and only if $\D$ is Cohen-Macaulay and $\D_V$ is Cohen-Macaulay for all subsets $V \subseteq [n]$ with $2 \le |V| \le \dim \D +1$.
\end{Theorem}

The Cohen-Macaulayness of all symbolic powers $I_\D^{(m)}$ can be characterized by means of matroid complexes.
Recall that a {\it matroid complex} is a collection of subsets of a finite set, called {\it independent sets}, with the following properties: \smallskip

(i) The empty set is independent.\par
(ii) Every subset of an independent set is independent. \par
(iii) If $F$ and $G$ are two independent sets and $F$  has more elements than $G$, then there exists an element in $F$ which is not in $G$ that when added to $G$ still gives an independent set. 

\begin{Theorem}\label{all} \cite[Theorem 3.5]{MT2}
$I_\D^{(m)}$ is a Cohen-Macaulay ideal for all $m \ge 1$ if and only if $\D$ is a matroid complex.
\end{Theorem}

If  $\dim \D = 2$, we can make these characterisations more precise. For that we shall need the following observation on the Cohen-Macaulayness of the union of two Cohen-Macaulay complexes.

\begin{Lemma}\label{union}
Let $\G_1$ and $\G_2$ be two Cohen-Macaulay complexes with $\dim \G_1 = \dim \G_2 = d \ge 1$. Then $\G_1 \cup \G_2$ is Cohen-Macaulay iff $\depth k[\G_1 \cap \G_2] \ge d$.
\end{Lemma}

\begin{proof} 
The assertion follows from the exact sequence
$$0 \to k[\G_1 \cup \G_2] \to k[\G_1]\oplus k[\G_2] \to k[\G_1 \cap \G_2] \to 0.$$
In fact, we have $\depth k[\G_1] = \depth k[\G_2] = d+1$ by the assumption. Therefore, $\depth k[\G_1 \cup \G_2] = d+1$ if and only if $\depth k[\G_1 \cap \G_2] \ge d$.
\end{proof}

Note that the condition $\depth k[\G_1 \cap \G_2] \ge 2$ just means that $\G_1 \cap \G_2$ is connected and $\dim \G_1 \cap \G_2 \ge 1$. \smallskip

For $F \in \D$ we denote by $\star F$ the subcomplex of $\D$ generated by the facets containing $F$ and call it the star of $F$ in $\D$. 
It is easy to see that $\D_V$ is the union of the stars of the faces of $\D$ with $|V|-1$ vertices in $V$.

\begin{Theorem}\label{second 2}
Let $\dim \D = 2$. Then $I_\D^{(2)}$ is a Cohen-Macaulay ideal if and only if the following conditions are satisfied:\par
{\rm (i) } $\D$ is Cohen-Macaulay,\par
{\rm (ii)} For every pair of vertices $u,v$, $\star\{u\} \cap \star\{v\}$ is a connected complex with dimension $\ge 1$,\par
{\rm (iii) } For every triple of vertices $u,v,w$ such that $\{u,v\},\{u,w\} \in \D$ and $\{v,w\} \not\in \D$, there exist a vertex $t$ such that $\{u,v,t\}, \{u,w,t\} \in \D$,\par
{\rm (iv)} For every triple of vertices $u,v,w$ such that $\{u,v\},\{u,w\},\{v,w\} \in \D$, $\{u,v,w\}$ $\in \D$ or there is a vertex $t$ such that $\{u,v,t\}, \{u,w,t\},\{v,w,t\} \in \D$.
\end{Theorem}

\begin{proof}
Assume that $\D$ satisfies the above conditions. By Theorem \ref{second 1},  $I_\D^{(2)}$ is a Cohen-Macaulay ideal if $\D_V$ is Cohen-Macaulay for all subsets $V \subseteq [n]$ with $2 \le |V| \le 3$. \par

If $V = \{u,v\}$, then $\D_V = \star\{u\} \cup \star\{v\}$. It is well known that the star of every face of a Cohen-Macaulay complex is Cohen-Macaulay. Therefore, we may apply Lemma \ref{union} to see that (i) and (ii) imply the Cohen-Macaulayness of $\D_V$.\par

If $V = \{u,v,w\}$, we may assume that $V$ has at least an edge in $\D$.  If $V$ has only an edge in $\D$, say $\{u,v\}$, then $\D_V = \star\{u,v\}$, which is Cohen-Macaulay by (i). 
If $V$ has two edges in $\D$, say $\{u,v\},\{u,w\}$, then $\D_V = \star\{u,v\}\cup \star\{u,w\}$.
Hence we can use Lemma \ref{union} to show that $\D_V$ is Cohen-Macaulay. Since $\star\{u,v\}$ and $\star\{u,w\}$ are Cohen-Macaulay, it suffices to show that $\star\{u,v\}\cap \star\{u,w\}$ is connected with dimension $\ge 1$. The connectedness follows from the fact that every face of $\star\{u,v\}\cap\, \star\{u,w\}$ contains $u$. By (iii), $\star\{u,v\}\cap\, \star\{u,w\}$ contains $\{u,t\}$, hence it has dimension $\ge 1$.
If $V$ has three edges in $\D$, then $\D_V = \star\{u,v\}\cup \star\{u,w\} \cup \star\{v,w\}$. Using Lemma \ref{union} and (iv) we can show similarly that $\star\{u,w\}\cup \star\{v,w\}$ is Cohen-Macaulay. Moreover,  $\star\{u,v\}\cap (\star\{u,w\} \cup \star\{v,w\})$ contains the facet $\{u,v,w\}$ or the edges $\{u,t\},\{v,t\}$. From this it follows that this complex is connected with dimension $\ge 1$. By Lemma \ref{union},  this implies the Cohen-Macaulayness of $\star\{u,v\}\cup \star\{u,w\} \cup \star\{v,w\}$.\smallskip

For the converse, assume that  $I_\D^{(2)}$ is a Cohen-Macaulay ideal. 
By Theorem \ref{second 1}, $\D$ is Cohen-Macaulay and $\D_V$ is Cohen-Macaulay for all subsets $V \subseteq [n]$ with $2 \le |V| \le 3$.  \par
If $V = \{u,v\}$, then  $\D_V = \star\{u\} \cup \star\{v\}$. Since $\D$ is Cohen-Macaulay, $\star\{u\}$ and $\star\{v\}$ are Cohen-Macaulay. By Lemma \ref{union}, the Cohen-Macaulayness of $\star\{u\} \cup \star\{v\}$ implies that $\star\{u\} \cap \star\{v\}$  is connected with dimension $\ge 1$.\par
If $V = \{u,v,w\}$ and $\{u,v\},\{u,w\} \in \D$ but $\{v,w\} \not\in \D$, then $\D_V = \star\{u,v\}\cup\, \star\{u,w\}$. 
Since $\star\{u,v\}$, $\star\{u,w\}$ and $\D_V$ are Cohen-Macaulay, we can use Lemma \ref{union} to deduce that $\star\{u,v\}\cap \star\{u,w\}$ is of dimension $\ge 1$. Since $u$ belongs to every face  of $\star\{u,v\}\cap \star\{u,w\}$, this complex must contain at least an edge, say $\{u,t\}$. Then  $\{u,v,t\}, \{u,w,t\} \in \D$. \par
If $V = \{u,v,w\}$  and $\{u,v\},\{u,w\},\{v,w\} \in \D$, $\D_V =  \star\{u,v\}\cup \star\{u,w\} \cup \{v,w\}$. Assume that $\{u,v,w\} \notin \D$.  If there doesn't exist any vertex $t$ such that $\{u,v,t\}, \{u,w,t\},\{v,w,t\} \in \D$, the geometric realization of $\D_V$ is homeomorphic to the triangle of the 
 edges $\{u,v\},\{u,w\},\{v,w\}$.  By \cite[Corollary 5.4.6]{BrH}, this implies that $\D_V$ is not Cohen-Macaulay, a contradiction.
\end{proof}

For $m \ge 2$ we have the following immediate consequence of Theorem \ref{all}.

\begin{Corollary} \label{matroid}
Let $\dim \D = 2$. Then $I_\D^{(m)}$ is a Cohen-Macaulay ideal for all $m \ge 1$ if and only if the following conditions are satisfied:\par
{\rm (i)} For every vertex $u$ and every edge $\{v,w\}$ not containing $u$ in $\D$, $\{u,v\} \in \D$ or $\{u,w\} \in \D$,\par
{\rm (ii)} For every face $\{u,v\}$ and every facet $\{u,w,t\}$ in $\D$, $\{u,v,w\} \in \D$ or $\{u,v,t\} \in \D$.
\end{Corollary}

\begin{proof} 
By \cite[Theorem 39.1]{Sch}, $\D$ is a matroid complex iff for every pair of faces $I$ and $J$ with $|I \setminus J| = 1$ and $|J\setminus I| = 2$, there is a vertex $x \in J \setminus I$ such that $I \cup\{x\}$ is a face of $\D$. Since $\dim \D =2$, $|J| \le 3$ so that we obtain conditions (i) and (ii).
\end{proof}

Now we will combine the results on the equality $I_\D^{(m)} = I_\D^m$ and the above characterizations of the Cohen-Macaulayness of $I_\D^{(m)}$ to see when $I_\D^m$ is Cohen-Macaulay for each $m \ge 2$. For that we need the following observation.

\begin{Lemma}\label{link} 
Let $\D$ be a simplicial complex and $F$ a face of $\D$.
If $I_\D^mR[x_i^{-1}|\ i \in F]$ is Cohen-Macaulay, then $I_{\lk F}^m$ is Cohen-Macaulay.
\end{Lemma}

\begin{proof}
Let $Y$ denote the set of the variables $x_i$ such that $i \not\in F$ and $i$ is not a vertex of $I_{\lk F}$.
It is easy to see that $I_\D^m R[x_i^{-1}|\ i \in F] = (I_{\lk G},Y)^mR[x_i^{-1}|\ i \in F]$.
Let $S$ be the polynomial ring over $K$ in the variables $x_i$, where $i$ is a vertex of $\lk F$. 
Then $I_{\lk F}$ is an ideal in $S$.
Let  $T$ be the polynomial ring over $K$ in the variables $x_i$, $i \not\in F$. Then $T = S[Y]$ and
$R[x_i^{-1}|\ i \in F] = T[x_i^{\pm 1}|\ i\in F]$. Hence, $(I_{\lk G},Y)^mT$ is Cohen-Macaulay.
On the other hand, $(I_{\lk G},Y)^mT = I_{\lk G}^mT + I_{\lk G}^{m-1}(Y)T + \cdots + (Y)^mT$.
From this it follows that $S/ I_{\lk G}^m$ is a direct summand of $T/(I_{\lk G},Y)^mT$ as an $S$-module.
Therefore,  $I_{\lk F}^m$ is Cohen-Macaulay.
\end{proof}

We shall also need the following description of one-dimensional complexes for which $I_\D^m$ is Cohen-Macaulay.

\begin{Lemma}\label{dim 1} \cite[Corollary 3.4 and Corollary 3.5]{MT1}
Let $\dim \D = 1$. Then \par
{\rm (i)} $I_\D^2$ is Cohen-Macaulay if and only if $\D$ is a path of length 2 or a cycle of length 4 or 5. \par
{\rm (ii)} $I_\D^m$ is Cohen-Macaulay for some $m \ge 3$ if and only if $\D$ is a path of length 2 or a cycle of length 4.
\end{Lemma}

\begin{Theorem}\label{CM 1}
Let $\D$ be a two-dimensional simplicial complex on $n \ge 5$ vertices. Then $I_\D^2$ is Cohen-Macaulay if and only if $\D$ is one of the following complexes up to a permutation of the vertices, \par
{\rm (i) } $n=5$ and 
\begin{align*}
\F(\D) & = \big\{\{1,3,4\},\{1,3,5\},\{2,3,4\},\{2,3,5\}\big\}\ \text{or}\\
\F(\D) & = \big\{ \{1,2,4\}, \{1,2,5\}, \{1,3,4\}, \{1,3,5\}, \{2,3,4\}, \{2,3,5\} \big\}. 
\end{align*} \par

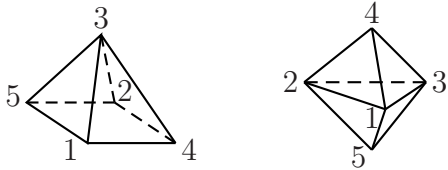
\begin{center}
\psset{unit=0.9cm}
\begin{pspicture}(-1.5,-0.8)(5.2,1.5)

\pspolygon(-1.1,0.1)(-0.2,-0.5)(0,1.1)
\psline[linestyle=dashed](0,1.1)(0.2,0.1)
\psline[linestyle=dashed](1.1,-0.5)(0.2,0.1)
\psline[linestyle=dashed](-1.1,0.1)(0.2,0.1)
\psline(1.1,-0.5)(-0.2,-0.5)
\psline(1.1,-0.5)(0,1.1) 

 \rput(0,1.35){\rm 3}
 \rput(-1.3,0.2){\rm 5}
 \rput(0.35,0.27){\rm 2}
 \rput(1.3,-0.6){\rm 4}
 \rput(-0.45,-0.6){\rm 1}
 
\psline(3,0.4)(4.2,0)
\psline(4.8,0.4)(4.2,0)
\psline[linestyle=dashed](3,0.4)(4.8,0.4) 
\psline(4,-0.6)(3,0.4)
\psline(4,-0.6)(4.8,0.4)
\psline(4,-0.6)(4.2,0)
\psline(4,1.2)(3,0.4)
\psline(4,1.2)(4.8,0.4)
\psline(4,1.2)(4.2,0)

\rput(4,1.42){\rm 4}
 \rput(3.8,-0.7){\rm 5}
 \rput(2.8,0.4){\rm 2}
 \rput(5,0.4){\rm 3}
 \rput(4,-0.13){\rm 1}

\end{pspicture}
\end{center} 

{\rm (ii) } $n=6$ and 
\begin{align*}
\F(\D) & =\big\{\{1,3,6\}, \{1,4,6\}, \{2,4,6\}, \{2,5,6\}, \{3,5,6\} \big\}\ \text{or}\\
\F(\D) & = \big\{\{1,2,3\}, \{1,2,6\},\{1,3,5\},\{1,5,6\}, \{2,3,4\},\{2,4,6\},\{3,4,5\},\{4,5,6\} \big\}\ \text{or}\\
\F(\D) & = \big\{ \{1,2,4\}, \{1,2,6\}, \{1,3,4\}, \{1,3,5\}, \{1,5,6\}, \{2,3,4\}, \{2,3,6\}, \{3,5,6\} \big\}.
\end{align*}

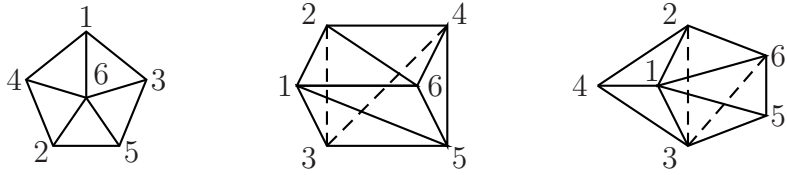
\begin{center}
\psset{unit= 0.8cm}
\begin{pspicture}(-1.5,-1.4)(12,1.4)
\pspolygon(-0.5,0.9)(0.5,0.1)(0.05,-1)(-1.05,-1)(-1.5,0.1)
\psline(-0.5,-0.2)(-0.5,0.9)
\psline(-0.5,-0.2)(0.5,0.1)
\psline(-0.5,-0.2)(-1.5,0.1)
\psline(-0.5,-0.2)(0.05,-1)
\psline(-0.5,-0.2)(-1.05,-1)
\rput(-0.5,1.15){\rm 1}
\rput(-1.7,0.1){\rm 4}
\rput(0.7,0.1){\rm 3}
\rput(-1.25,-1.1){\rm 2}
\rput(0.25,-1.1){\rm 5}
\rput(-0.25,0.2){\rm 6} 

\pspolygon(3,0)(3.5,1)(5,0)
\pspolygon(3,0)(3.5,-1)(5.5,-1)
\pspolygon(5,0)(5.5,1)(5.5,-1)
\psline(3.5,1)(5.5,1)
\psline(3,0)(5,0)
\psline[linestyle=dashed](3.5,1)(3.5,-1)
\psline[linestyle=dashed](3.5,-1)(5.5,1)
\rput(2.8,0){\rm 1}
\rput(3.2,1.2){\rm 2}
\rput(3.2,-1.2){\rm 3}
\rput(5.3,0){\rm 6}
\rput(5.7,-1.2){\rm 5}
\rput(5.7,1.2){\rm 4} 

\pspolygon(9,0)(9.5,1)(10.8,0.5)
\pspolygon(9,0)(9.5,-1)(10.8,-0.5)
\psline(10.8,0.5)(10.8,-0.5)
\psline(8,0)(9,0)
\psline(8,0)(9.5,1)
\psline(8,0)(9.5,-1)
\psline[linestyle=dashed](9.5,1)(9.5,-1)
\psline[linestyle=dashed](9.5,-1)(10.8,0.5)
\rput(8.9,0.25){\rm 1}
\rput(9.2,1.2){\rm 2}
\rput(9.2,-1.2){\rm 3}
\rput(11,0.5){\rm 6}
\rput(11,-0.5){\rm 5}
\rput(7.7,0){\rm 4} 
\end{pspicture}
\end{center}

{\rm (iii) } $n=7$ and 
\begin{align*} 
\F(\D) & = \big\{\{1,3,6\}, \{3,5,6\}, \{2,5,6\}, \{2,4,6\}, \{1,4,6\}, \\
& \hspace{8mm}  \{1,3,7\}, \{3,5,7\}, \{2,5,7\}, \{2,4,7\}, \{1,4,7\}\big\}\ \text{or} \\
\F(\D) &= \big\{ \{1,2,4\}, \{1,2,5\}, \{1,3,4\}, \{1,3,6\}, \{1,5,6\},  \\
& \hspace{8 mm} \{2,3,4\}, \{2,3,7\}, \{2,5,7\}, \{3,6,7\}, \{5,6,7\} \big\}.
\end{align*}

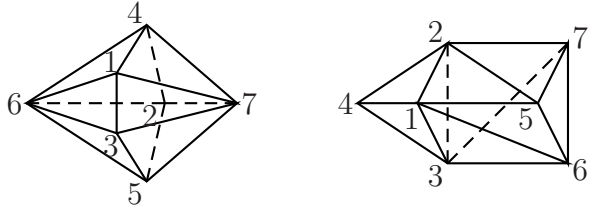
\begin{center}
\psset{unit=0.8cm}
\begin{pspicture}(-0.2,-1.6)(8.2,1.6)
\pspolygon(-0.5,0)(1,0.5)(1,-0.5)
\pspolygon(1,-0.5)(1.5,-1.3)(3,0)
\pspolygon(1,0.5)(1.5,1.3)(3,0)
\psline(-0.5,0)(1.5,1.3)
\psline(-0.5,0)(1.5,-1.3)
\psline[linestyle=dashed](1.8,0)(-0.5,0)
\psline[linestyle=dashed](1.8,0)(3,0)
\psline[linestyle=dashed](1.8,0)(1.5,1.3)
\psline[linestyle=dashed](1.8,0)(1.5,-1.3)
\rput(-0.7,0){\rm 6}
\rput(0.9,0.7){\rm 1}
\rput(0.9,-0.7){\rm 3}
\rput(1.3,-1.5){\rm 5}
\rput(1.3,1.5){\rm 4}
\rput(1.55,-0.2){\rm 2} 
\rput(3.2,0){\rm 7} 

\pspolygon(5,0)(6,0)(6.5,1)
\pspolygon(6,0)(6.5,-1)(8.5,-1)
\pspolygon(6.5,1)(8.5,1)(6.5,1)
\pspolygon(8,0)(8.5,1)(8.5,-1)
\psline(6.5,1)(8,0)
\psline(6,0)(8,0)
\psline(5,0)(6.5,-1)
\psline[linestyle=dashed](6.5,1)(6.5,-1)
\psline[linestyle=dashed](6.5,-1)(8.5,1)
\rput(4.8,0){\rm 4}
\rput(5.9,-0.25){\rm 1}
\rput(6.3,1.2){\rm 2}
\rput(6.3,-1.2){\rm 3}
\rput(7.8,-0.25){\rm 5}
\rput(8.7,1.1){\rm 7} 
\rput(8.7,-1.1){\rm 6} 
\end{pspicture}
\end{center}

{\rm (iv) } $n=8$ and
\begin{align*}
\F(\D) & =  \big\{\{1,3,6\},\{1,3,8\},\{1,4,7\},\{1,4,8\},\{1,6,7\},\{2,4,7\},\\
& \hspace{8mm}\{2,4,8\},\{2,5,6\},\{2,5,8\},\{2,6,7\},\{3,5,6\},\{3,5,8\}\big\}.
\end{align*}

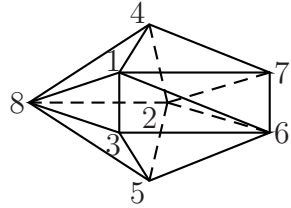
\begin{center}
\psset{unit=0.8cm}
\begin{pspicture}(-0.2,-1.6)(3.2,1.6)
\pspolygon(-0.5,0)(1,0.5)(1,-0.5)
\pspolygon(1,-0.5)(1.5,-1.3)(3.5,-0.5)
\pspolygon(1,0.5)(1.5,1.3)(3.5,0.5)
\psline(3.5,-0.5)(3.5,0.5)
\psline(1,0.5)(3.5,-0.5)
\psline(-0.5,0)(1.5,1.3)
\psline(-0.5,0)(1.5,-1.3)
\psline[linestyle=dashed](1.8,0)(-0.5,0)
\psline[linestyle=dashed](1.8,0)(3.5,0.5)
\psline[linestyle=dashed](1.8,0)(3.5,-0.5)
\psline[linestyle=dashed](1.8,0)(1.5,1.3)
\psline[linestyle=dashed](1.8,0)(1.5,-1.3)
\rput(-0.7,0){\rm 8}
\rput(0.9,0.7){\rm 1}
\rput(0.9,-0.7){\rm 3}
\rput(1.3,-1.5){\rm 5}
\rput(1.3,1.5){\rm 4}
\rput(1.5,-0.25){\rm 2} 
\rput(3.7,-0.5){\rm 6} 
\rput(3.7,0.5){\rm 7} 
\end{pspicture}
\end{center}
\end{Theorem}

\begin{proof}
Assume that $I_\D^2$ is Cohen-Macaulay. Then $I_\D^{(2)} = I_\D^2$ and $I_\D^{(2)}$ is Cohen-Macaulay. We distinguish two cases. \smallskip

{\em Case 1}: Every triangle of $\D_1$ is a facet of $\D.$ \smallskip

If $\D_1$ contains a complete subgraph on a set $V$ of 4 vertices, every triangle of $V$ is a facet of $\D$. Let $i \in [n] \setminus V,$ then $V \cup \{i\}$ can't be divided into two nonfaces of size two and three. Thus, we have $\prod_{j \in V \cup \{i\}}x_j \in I_\D^{(2)}\setminus I_\D^2,$ a contradiction. Therefore, $\D_1$ doesn't contain $K_4$. Together with Lemma \ref{nontriangle}, this implies that $\D_1$ is a Ramsey (4,3)-graph. Hence $n \le 8$ \cite{Ra}. \par

Let $n = 5$. By Theorem \ref{5}  we may assume that $\{1,2\},\{3,4,5\} \notin \D$.
Since the vertices 1,2 must belong to some facets of $\D$, we may also assume $\{1,3,4\}, \{2,3,5\} \in \D$. Since $\{3,4,5\}$ can't be a triangle of $\D_1$, $\{4,5\} \notin \D$. 
If $\{2,3,4\} \notin \D$, then $\star\{2\} \cap \star\{4\}$ is generated by the vertex $3$, a contradiction to Theorem \ref{second 2} (ii). Therefore, $\{2,3,4\} \in \D$. Similarly, we also have $\{1,3,5\} \in \D$. Since there are no further possibilities for facets of $\D$,
$$\F(\D) = \big\{\{1,3,4\},\{1,3,5\},\{2,3,4\},\{2,3,5\}\big\}.$$
In this case, $I_\D = (x_1x_2,x_4x_5)$, which is a complete intersection. Hence $I_\D^m$ is Cohen-Macaulay for all $m \ge 1$. \par

Let $n=6.$ If $\overline{\D_1}$ has an induced cycle of length 5, say on the ordered vertices $1,2,3,4,5,$ then $1,3,5,2,4$ are the ordered vertices of an induced cycle $C$ of length 5 of $\D_1$. Every edge of $C$ must belong to a facet containing 6.  
From this it follows that $\D$ is the cone over $C$ with:
$$\F(\D)=\big\{\{1,3,6\}, \{1,4,6\}, \{2,4,6\}, \{2,5,6\}, \{3,5,6\} \big\}.$$
In this case, we may also consider $I_\D$ as the Stanley-Reisner ideal of $C$. Hence $I_\D^2$ is Cohen-Macaulay but $I_\D^m$ is not Cohen-Macaulay for all $m \ge 3$ by Lemma \ref{dim 1}. \par

If $\overline{\D_1}$ doesn't have an induced cycle of length 5, then $\overline{\D_1}$ is a bipartite graph because $\overline{\D_1}$ doesn't contain $K_3$ by Theorem \ref{case 2}. Since $\D_1$ doesn't contain $K_4$, the maximal size of an independent set of $\overline{\D_1}$ is $\le 3$. 
Note that the complement of an independent set is a vertex cover. Then the minimal size of a vertex cover of $\overline \D_1$ is $\ge 3$.
By K\"onig's theorem for a bipartite graph, the minimal size of a vertex cover equals  the maximal size of a matching. Therefore, $\overline \D_1$ has a matching of 3 edges, say $\{1,4\},\{2,5\},\{3,6\}$. Furthermore, since the vertices of the bipartite graph $\overline{\D_1}$ can be divided into two independent sets of size 3, we may assume that $\{1,2,3\},\{4,5,6\}$ are triangles of $\D_1$. From this it follows that $\{1,2,3\},\{4,5,6\} \in \D$. If $\{1,5\} \notin \D$, then $\star\{1\} \cap \star\{5\}$ is contained in the zero-dimensional complex generated by the vertices $3,6$, which contradicts Theorem \ref{second 2} (ii). Thus, $\{1,5\} \in \D$. Similarly, we also have $\{1,6\},\{2,4\},\{2,6\},\{3,4\},\{3,5\} \in \D$.
Hence 
$$\{1,2,6\},\{1,3,5\},\{1,5,6\}, \{2,3,4\},\{2,4,6\},\{3,4,5\} \in \D.$$
Since there are no further possibilities for the faces of $\D$, we can conclude that
$$\mathcal{F}(\Delta)= \big\{\{1,2,3\}, \{1,2,6\},\{1,3,5\},\{1,5,6\}, \{2,3,4\},\{2,4,6\},\{3,4,5\},\{4,5,6\} \big\}.$$
In this case, $I_\D = (x_1x_4,x_2x_5,x_3x_6)$, which is a complete intersection. Hence $I_\D^m$ is Cohen-Macaulay for all $m \ge 1$. \par

Let $n = 7.$ By Lemma \ref{Ramsey properties}, $\overline{\D_1}$ has an induced cycle of length 5 or 7. \par

If $\overline{\D_1}$ has an induced cycle of length 5, say on the ordered vertices $1,2,3,4,5$. Then $1,3,5,2,4$ are the ordered vertices of an induced cycle $C$ of length 5 of $\D_1$.  If  $\{6,7\} \in \D$, we may assume that $\{1,6,7\} \in \D$. Then $\lk \{1\}$ has 4 vertices $3,4,6,7$.  By Lemma \ref{link}, $I_{\lk \{1\}}^2$ is Cohen-Macaulay. By Lemma \ref{dim 1}, $\lk \{1\}$ must be a cycle of length 4. Since $\{6,7\} \in \lk \{1\}$, this implies $\{3,4\} \in \lk \{1\}$,  a contradiction.  So $\{6,7\} \not\in \D$. Hence $\lk \{6\}$ is a subgraph of the cycle $C$. By Lemma \ref{link}, $I_{\lk \{6\}}^2$ is Cohen-Macaulay. By Lemma \ref{dim 1}, $\lk \{6\}$ can be only a path of length 2 or the cycle  $C$. If $\lk \{6\}$ is a path of length 2, say on the vertices $1,3,4$, then $\star\{2\} \cap \star\{6\}$ is generated by the vertex $4$, a contradiction to Theorem \ref{second 2} (ii). So $\lk \{6\} = C$. Similarly, we also have $\lk \{7\} = C$. From this it follows that 
\begin{align}
\F(\D) &= \big\{ \{1,3,6\}, \{3,5,6\}, \{2,5,6\}, \{2,4,6\}, \{1,4,6\}, \notag \\
& \hspace{8 mm} \{ \{1,3,7\}, \{3,5,7\}, \{2,5,7\}, \{2,4,7\}, \{1,4,7\}\big\}. \notag
\end{align}
Using Theorem \ref{case 2} and Theorem \ref{second 2} we can verify that $I_\D^2$ is Cohen-Macaulay.
By Theorem \ref{case 3} and Theorem \ref{case 4}, $I_\D^{(m)} \neq I_\D^m$ for all $m \geq 3$. 
Hence $I_\D^m$ is not Cohen-Macaulay for all $m \geq 3.$ \par

If  $\overline{\D_1}$ has an induced cycle of length 7, then we must have
$$\F(\D)  = \big\{\{1,3,5\},\{1,3,6\}, \{1,4,6\}, \{2,4,6\},\{2,4,7\},\{2,5,7\},\{3,5,7\} \big\}.$$
We see that $\star \{1\} \cap \star \{2\}$ is generated by the vertex 5 and $\{4,6\}.$ Hence $I_\D^{(2)}$ is not Cohen-Macaulay by Theorem \ref{second 2} (ii). By \cite[Corollary 4.4]{MT2}, the non-Cohen-Macaulayness of $I_\D^{(2)}$ implies the non-Cohen-Macaulayness of $I_\D^{(m)}$ for all $m \ge 2$. Therefore, $I_\D^m$ is not Cohen-Macaulay for all $m \ge 2$.  \par

If $n = 8$,  $\overline{\D_1}$ has an induced cycle of length 5, say on the ordered vertices $1,2,3,4,5$  by Lemma \ref{Ramsey properties}.  Then $1,3,5,2,4$ are the ordered vertices of a cycle of length 5 of $\D_1$. Since $\overline{\D_1}$ doesn't contain $K_3$, we may assume that $\{6,7\} \in \D$. Furthermore,  each of the vertices 6 and 7 must be adjacent to at least one vertex of any edge of $\overline{\D_1}$. From this it follows that 6 and 7 are adjacent to at least 3 vertices among 1,2,3,4,5. Hence we can find a vertex, say 1 which is adjacent to both $6$ and $7$. Since $1,6,7$ form a triangle of $\D_1$, $\{1,6,7\}$ is a facet of $\D$ by the assumption of Case 1. It follows that $\lk \{1\}$ contains two non-adjacent vertices 3,4 and $\{6,7\}$. By Lemma \ref{link}, $I_{\lk \{1\}}^2$ is Cohen-Macaulay. By Lemma \ref{dim 1}, $\lk \{1\}$ must be an induced cycle of length 5. The fifth vertex of this cycle must be 8. Without restriction we may assume that this cycle has the ordered vertices $3,6,7,4,8$. Then $\{1,3,6\},\{1,3,8\},\{1,4.7\},\{1,4,8\}$ are facets of $\D$. 
Since $\{3,7\}$ is not an edge of this cycle,  $\{1,3,7\} \notin \D$. Hence $\{3,7\} \notin \D$. Therefore, $\lk \{3\}$ has four vertices 1,5,6,8. By Lemma \ref{link} and Lemma \ref{dim 1}, $\lk \{3\}$ must be a cycle of length 4. Since $\{1,5\}\notin \D$, this cycle has the ordered vertices 1,6,5,8. Hence $\{3,5,6\}, \{3,5,8\}$ are facets of $\D$.
Similarly, if we consider $\lk\{4\}$, we see that $\{2,4,7\},\{2,4,8\}$ are facets of $\D$. 
If $\{2,6\},\{5,7\} \in \D$, we would have $\{2,6,7\},\{5,6,7\} \in \D$,  hence $\{6,7\}$ belongs to three facets of $\D$, a contradiction to Lemma \ref{join face}. So we have $\{2,6\} \notin \D$ or $\{5,7\} \notin \D$. Without loss of generality we may assume that $\{5,7\} \notin \D$. Then $\lk \{5\}$ has four vertices 2,3,6,8. By Lemma \ref{link} and Lemma \ref{dim 1}, $\lk \{5\}$ must be a cycle of length 4. Since $\{2,3\} \not\in \D$, this cycle has the ordered vertices 2,6,3,8. Hence $\{2,5,6\},\{2,5,8\}$ are facets of $\D$. Now, since $2,6,7$ form a triangle of $\D_1$, $\{2,6,7\}$ is a facet of $\D$.  Using Lemma \ref{join face} we can see that there are no further facets of $\D$.
Therefore,
\begin{align*}
\F(\D) & =  \big\{\{1,3,6\},\{1,3,8\},\{1,4,7\},\{1,4,8\},\{1,6,7\},\{2,4,7\},\\
& \hspace{0.8cm}\{2,4,8\},\{2,5,6\},\{2,5,8\},\{2,6,7\},\{3,5,6\},\{3,5,8\}\big\}.
\end{align*}
Using Theorem \ref{case 2} and Theorem \ref{second 2} we can verify that $I_\D^2$ is Cohen-Macaulay.
By Theorem \ref{case 3} and Theorem \ref{case 4}, $I_\D^{(m)} \neq I_\D^m$ for all $m \geq 3$. 
Hence $I_\D^m$ is not Cohen-Macaulay for all $m \geq 3.$ \smallskip

{\em Case 2}: $\D_1$ contains a triangle, say $\{1,2,3\},$ which does not belong to $\D.$ \smallskip

By Theorem \ref{second 2} (iv), there is a vertex, say $4$ such that $\{1,2,4 \},$ $\{1,3,4\},$ $\{2,3,4 \} \in \D.$
Let $i \neq 1,2,3,4.$ By Theorem \ref{case 2} (iii), at least one of the sets $\{1,2,i\},$ $\{1,3,i\}$ or $\{2,3,i \}$ must be a facet of $\D.$  By Lemma \ref{join face}, each edge of $\{1,2,3\}$  belong to at most two facets of $\D$. Hence for each edge of $\{1,2,3\}$, 
there is at most one vertex outside $\{1,2,3,4\}$ which together with the edge forms a facet of $\D$.  From this it follows that there are at most three vertices $i \neq 1,2,3,4$ so that $n \leq 7.$ \par

Let $n=5.$ By Theorem \ref{5}, we must have $\{4,5\} \notin \D$. If $\{1,5\} \notin \D$,  $\star\{1\} \cap \star\{5\}$ is contained in the zero-dimensional complex generated by the vertices 2,3, which contradicts Theorem \ref{second 2} (ii). So $\{1,5\} \in \D$. Similarly, we also have $\{2,5\} \in \D$. Hence $\{1,2,5\}$ is a triangle of $\D$. Since there is no vertex $i \neq 1,2,5$ such that $\{1,2,i\},$ $\{1,5,i\}, $ $\{2,5,i\} \in \D,$ it follows from Theorem \ref{second 2} (iv)  that $\{1,2,5\} \in \D.$ Similarly, we also have $\{1,3,5\},\{2,3,5\} \in \D$. Since there are no further possibilities for facets of $\D$, 
$$\F(\D)= \big\{ \{1,2,4\}, \{1,2,5\}, \{1,3,4\}, \{1,3,5\}, \{2,3,4\}, \{2,3,5\} \big\}. $$
In this case, $I_\D = (x_1x_2x_3,x_4x_5)$, which is a complete intersection. Thus, $I_\D^m$ is a Cohen-Macaulay ideal for all $m \geq 2.$ \par

Let $n=6.$ Without restriction we may assume $\{1,2,6\}, \{1,3,5\} \in \D.$
Applying Theorem \ref{case 2} (iv) to the 5-sets containing $1,2,3,4$ we see that $\{4,5\}, \{4,6\} \notin \D$. By Theorem \ref{case 2} (i), this implies $\{5,6\} \in \D$. Hence $\{1,5,6\}$ is a triangle of $\D_1$. Since $\{1,4,5\}, \{1,4,6\}, \{4,5,6\} \notin \D$,  $\{1,5,6\} \in \D$ by Theorem \ref{case 2} (iii). By Lemma \ref{join face},  $\{1,2,5\},\{1,3,6\} \notin \D.$ If $\{2,5\},$ $\{3,6\} \notin \D$, then $\star\{2\} \cap \star\{5\}$ is generated by the vertex 3 and $\{1,6\}$, which contradicts Theorem \ref{second 2} (ii). So we may assume that $\{3,6\} \in \D$. 
Now, applying Theorem \ref{case 2} (iv) to the vertices $1,2,3,5,6$ we can easily verify that $\{2,5\} \notin \D.$ Since $\{2,3,6\}$  and $\{3,5,6\}$ are triangles of $\D_1$ and since there are no cones over these triangles, we must have $\{2,3,6\}, \{3,5,6\} \in \D$ by Theorem \ref{second 2} (iv). Now using Lemma \ref{join face} we can check that there are no further facets of $\D$. Therefore,
$$\F(\D)= \big\{ \{1,2,4\}, \{1,2,6\}, \{1,3,4\}, \{1,3,5\}, \{1,5,6\}, \{2,3,4\}, \{2,3,6\}, \{3,5,6\} \big\}.$$
Using Theorem \ref{case 2} and Theorem \ref{second 2} we can easily check that $I_\D^2$ is Cohen-Macaulay.
By Theorem \ref{case 3} and Theorem \ref{case 4}, $I_\D^{(m)} \neq I_\D^m$ for all $m \geq 3$. Hence $I_\D^m$ is not Cohen-Macaulay for all $m \geq 3.$ \par

Let $n=7.$ Without restriction we may assume that $\{1,2,5\},$ $\{1,3,6\},$ $\{2,3,7\} \in \D.$ 
Applying Theorem \ref{case 2} (iv) to the 5-sets containing $1,2,3,4$ we see that $\{4,5\},$ $\{4,6\},$ $\{4,7\} \notin \D$. Hence $\{5,6\},\{5,7\},\{6,7\} \in \D$ by Theorem \ref{case 2} (i). It follows that $\{1,5,6\}, \{2,5,7\},\{ 3,6,7\}, \{5,6,7\}$ are triangles of $\D_1$. Since there are no facets containing the vertex 4 and an edge of these triangles, Theorem \ref{case 2} (iii) implies $\{1,5,6\}, \{2,5,7\}$, $\{3,6,7\}, \{5,6,7\} \in \D.$
Using Lemma \ref{join face} we can verify that there are no further possibilities for facets of $\D$. Therefore,
\begin{align*}
\F(\D) &= \big\{ \{1,2,4\}, \{1,2,5\}, \{1,3,4\}, \{1,3,6\}, \{1,5,6\},  \\
& \hspace{8 mm} \{2,3,4\}, \{2,3,7\}, \{2,5,7\}, \{3,6,7\}, \{5,6,7\} \big\}.
\end{align*}
Using Theorem \ref{case 2} and Theorem \ref{second 2} we can easily check that $I_\D^2$ is Cohen-Macaulay.
By Theorem \ref{case 3} and Theorem \ref{case 4}, $I_\D^{(m)} \neq I_\D^m$ for all $m \geq 3$. Hence $I_\D^m$ is not Cohen-Macaulay for all $m \geq 3.$ \smallskip

Summing up the above analysis we obtain the assertions of Theorem \ref{CM 1}.
\end{proof}

In checking the Cohen-Macaulayness of $I_\D^2$ for some complexes in the above proof we have to test the Cohen-Macaulayness of these complexes by Theorem \ref{second 2}. By a result of Munkres \cite[Corollary 5.4.6]{BrH}, the Cohen-Macaulayness of these complexes follows from the fact that their geometric realization are homeomorphic to a sphere as can be seen from the pictures of Theorem \ref{CM 1}. Thus, the Cohen-Macaulayness of $I_\D^2$ doesn't depend on the characteristic of the base field for two-dimensional complexes. This displays a different behavior than the Cohen-Macaulayness of $I_\D$ \cite[Section 5.3]{BrH}. \smallskip

The proof of  Theorem \ref{CM 1} also gives information on  the Cohen-Macaulayness of $I_\D^m$ for $m\ge 3$. 

\begin{Theorem} \label{CM 2}
Let $\D$ be a two-dimensional simplicial complex on $n \ge 5$ vertices. Then $I_\D^m$ is Cohen-Macaulay for some $m \ge 3$ resp. for all $m \ge 1$ if and only if $\D$ is one of the following complexes up to a permutation of the vertices, \par
{\rm (i) } $n=5$ and 
\begin{align*}
\F(\D) & = \big\{\{1,3,4\},\{1,3,5\},\{2,3,4\},\{2,3,5\}\big\}\ \text{or}\\
\F(\D) & = \big\{ \{1,2,4\}, \{1,2,5\}, \{1,3,4\}, \{1,3,5\}, \{2,3,4\}, \{2,3,5\} \big\}. 
\end{align*} \par
{\rm (ii) } $n=6$ and 
$$\F(\D) = \big\{\{1,2,3\}, \{1,2,6\},\{1,3,5\},\{1,5,6\}, \{2,3,4\},\{2,4,6\},\{3,4,5\},\{4,5,6\} \big\}.$$
\end{Theorem}

\begin{proof} 
By \cite[Corollary 4.4]{MT2} and Corollary \ref{preservation} (i), if $I_\D^m$ is Cohen-Macaulay for some $m \ge 3$, then  $I_\D^2$ is Cohen-Macaulay. Thus, we only need to check for which case of Theorem \ref{CM 1},  $I_\D^m$ is Cohen-Macaulay for some $m \ge 3$ resp. for all $m \ge 1$.  But this has been done in the proof of Theorem \ref{CM 1}.
\end{proof}

By \cite{CN}, $I_\D$ is a complete intersection if and only if $I_\D^m$ is Cohen-Macaulay for all large $m$ (or all $m \ge 1$). We can improve this result as follows.

\begin{Corollary}
Let $\D$ be a two-dimensional simplicial complex on $n \ge 5$ vertices. Then $I_\D$ is a complete intersection if and only if $I_\D^m$ is Cohen-Macaulay for some $m \ge 3$.
\end{Corollary}

\begin{proof}
It suffices to show that $I_\D$ is a complete intersection in all cases of Theorem \ref{CM 2}.
In fact, we have up to a permutation of the variables,
\begin{align*}
I_\D & = (x_1x_2,x_4x_5)\ \text{or}\\
I_\D & =  (x_1x_2x_3,x_4x_5)\ \text{or}\\
I_\D & = (x_1x_4,x_2x_5,x_3x_6).
\end{align*}
\end{proof}

The above results show that there are complexes for which $I_\D^2$ is Cohen-Macaulay but $I_\D^m$ is not Cohen-Macaulay for all $m \ge 3$ and that if $I_\D^m$ is Cohen-Macaulay for some $m \ge 3$, then $I_\D$ is a complete intersection. Since this phenomenon also holds in the case $\dim \D = 1$,  it is quite natural to ask whether the same also holds in general.

\end{document}